\documentclass[11pt]{article}
\usepackage[margin=1.0in]{geometry}

\usepackage{verbatim}
\usepackage{hyperref}
\usepackage{amsmath}
\usepackage{amsthm}
\usepackage{amssymb}
\usepackage{amsfonts}
\usepackage{multicol}
\usepackage{subfigure}
\usepackage[numbers,sort&compress]{natbib}
\usepackage{booktabs}
\usepackage{titlesec}
\usepackage{longtable,tabu}

\usepackage{enumitem}
\setlist[enumerate]{label={\arabic*.}}

\usepackage{color}

\newcommand{\Hstar}{{\cal H}^{*}}
\newcommand{\calH}{{\cal H}}

\usepackage{tikz}
\usetikzlibrary{calc}

\usepackage{tkz-graph}

\newtheorem{theorem}{Theorem}[section]
\newtheorem{lemma}[theorem]{Lemma}

\theoremstyle{definition}

\newtheorem{problem}{Open Problem}

\newcommand{\squishlist}{
 \begin{list}{$\bullet$}
  { \setlength{\itemsep}{0pt}
     \setlength{\parsep}{3pt}
     \setlength{\topsep}{3pt}
     \setlength{\partopsep}{0pt}
     \setlength{\leftmargin}{2.5em}
     \setlength{\labelwidth}{1em}
     \setlength{\labelsep}{0.5em} } }

\newcommand{\squishlisttwo}{
 \begin{list}{$\triangleright$}
  { \setlength{\itemsep}{0pt}
     \setlength{\parsep}{0pt}
    \setlength{\topsep}{0pt}
    \setlength{\partopsep}{0pt}
    \setlength{\leftmargin}{2em}
    \setlength{\labelwidth}{1.5em}
    \setlength{\labelsep}{0.5em} } }

\newcommand{\squishend}{
  \end{list}  }

\begin{document}

\title{A refinement on the structure of vertex-critical ($P_5$, gem)-free graphs}
\author{
Ben Cameron\\ 
\small Department of Computing Science\\
\small The King's University\\
\small Edmonton, AB Canada\\
\small ben.cameron@kingsu.ca\\
\and
Ch\'{i}nh T. Ho\`{a}ng\\ 
\small Department of Physics and Computer Science\\
\small Wilfrid Laurier University\\
\small Waterloo, ON Canada\\
\small choang@wlu.ca\\
}

\date{\today}

\maketitle

\begin{abstract}
We give a new, stronger proof that there are only finitely many  $k$-vertex-critical ($P_5$,~gem)-free graphs for all $k$. Our proof further refines the structure of these graphs and allows for the implementation of a simple exhaustive computer search to completely list all $6$- and $7$-vertex-critical $(P_5$, gem)-free graphs. Our results imply the existence of polynomial-time certifying algorithms to decide the $k$-colourability of $(P_5$, gem)-free graphs for all $k$ where the certificate is either a $k$-colouring or a $(k+1)$-vertex-critical induced subgraph. Our complete lists for $k\le 7$ allow for the implementation of these algorithms for all $k\le 6$.
\end{abstract}

\section{Introduction}
For a fixed $k\ge 3$, it has long been known that determining the $k$-colourability of a general graph is an NP-complete problem~\cite{Karp1972}. When restricted to families of graphs, however, polynomial-time algorithms can be developed. For example, $k$-colourability can be decided in polynomial-time for all $k$ for perfect graphs~\cite{Grotschel1984} and $P_5$-free graphs~\cite{Hoang2010}. The latter is of particular interest because deciding $k$-colourability of $H$-free graphs remains NP-complete if $H$ contains an induced cycle~\cite{KaminskiLozin2007,Maffray1996} or claw~\cite{Holyer1981,LevenGail1983}. Further, the problem remains NP-complete for $P_6$-free graphs for all $k\ge 5$ and for $P_7$-free graphs for all $k\ge 4$~\cite{Huang2016}. Completing the complexity dichotomy for $4$-colourability of $H$-free graphs when $H$ is connected, is the polynomial-time algorithm to decide $4$-colourability of $P_6$-free graphs~\cite{P6freeconf,P6free1,P6free2}. Thus, $P_5$ is the largest connected graph whose forbidding results in the existence (assuming P$\neq$NP) of polynomial-time $k$-colourability algorithms for all $k$. While many open problems remain on the complexity of $k$-colouring $H$-free graphs, including $3$-colouring $P_8$-free graph~\cite{Bonomo2018} and $6$-colouring $(P_3+P_2)$-free graphs~\cite{HajebiLiSpirkl2021, Golovach2017}, the focus of this paper is to further the study of polynomial-time \textit{certifying} $k$-colourability algorithms. An algorithm is certifying if it returns an easily verifiable witness with each output to certify its correctness. A certifying algorithm offers the considerable advantage of an extra layer of certainty for users that the algorithm was correctly implemented. This has led to an increase in the development in certifying algorithms in the last decade and even to calls from some that only certifying algorithms should be developed for solving complex tasks~\cite{McConnell2011}. For a $k$-colourability algorithm to be certifying, it must return a $k$-colouring if one exists and a $(k+1)$-vertex-critical induced subgraph if the graph is not $k$-colourable. In a family of graphs, if there are only finitely many $(k+1)$-vertex critical graphs, then a polynomial-time $k$-colourability algorithm that certifies negative answers can be developed by determining if any of the vertex-critical graphs are induced subgraphs of the input graph (see~\cite{P5banner2019} for more details).

The classification of all 12 $4$-vertex-critical $P_5$-free graphs was used to develop a linear-time certifying $3$-colourability algorithm for $P_5$-free graphs~\cite{Bruce2009,MaffrayMorel2012}. For all $k\ge 4$, however, there are infinitely $(k+1)$-vertex-critical $P_5$-free graphs~\cite{Hoang2015}. This has led to tremendous interest in determining for which subfamilies of $P_5$-free graphs the polynomial-time $k$-colourability algorithms can be extended to certifying ones. Perhaps the most comprehensive known result is the dichotomy for $H$ of order $4$ that in all cases except $H=2K_2$ or $H=K_3+P_1$, there are only a finite number of $k$-vertex critical $(P_5,H)$-free graphs for all $k$~\cite{KCameron2021}. The open problem of completing such a dichotomy for $H$ of order $5$ was also posed in~\cite{KCameron2021}. Toward resolving this open problem, it is known that there are only finitely many $k$-vertex critical $(P_5,H$)-free for $k\le 5$ when $H=C_5$~\cite{Hoang2015} and for all $k$ when:

\begin{itemize}
\item $H$=banner~\cite[Theorem 3(i)]{Brause2022}
\item $H=K_{2,3}$ and $H=K_{1,4}$~\cite{Kaminski2019}
\item $H=P_2+3P_1$~\cite{CameronHoangSawada2022}
\item $H=P_3+2P_1$~\cite{AbuadasCameronHoangSawada2022}
\item $H=\overline{P_5}$~\cite{Dhaliwal2017}
\item $H=\overline{P_3+P_2}$ and $H=$gem~\cite{CaiGoedgebeurHuang2021}
\end{itemize}

Of particular interest among these results are the last two bullet points, as complete lists of all such $k$-vertex-critical graphs were given for $k=4$ and $k=5$~\cite{Dhaliwal2017,CaiGoedgebeurHuang2021}. These complete lists are essential for implementing the corresponding certifying algorithms. In this paper, we present a new proof that there are only finitely $k$-vertex-critical $(P_5,\text{ gem})$-free graphs. The proof in~\cite{CaiGoedgebeurHuang2021} involves bounding the order of every $k$-vertex-critical graph in one difficult structural case (this case is that the graph belongs to $\calH$, which will be defined later). Our new proof shows the stronger result that there are in fact no vertex-critical graphs in this difficult structural case. Our stronger result allows us to use a simple computer search to completely characterize all such graphs for $k=6$ and $k=7$, allowing for the implementation of polynomial-time certifying $5$- and $6$-colourability algorithms. Beyond their use for implementing certifying colourability algorithms, vertex-critical graphs are of interest in many other areas. In particular, vertex-critical $(P_5,\text{ gem})$-free graphs were very recently used to prove the Borodin-Kostochka Conjecture for $(P_5,\text{ gem})$-free graphs~\cite{CranstonLafayetteRabern2022}. Our proof that there are only finitely many $k$-vertex-critical $(P_5,\text{ gem})$-free graphs for all $k$ would allow for a simpler proof of the Main Theorem in~\cite{CranstonLafayetteRabern2022}, in particular, Case 11, the most involved case, is no longer required.

The rest of the paper is organized as follows: we first introduce some notation and definitions before proving our result on $k$-vertex-critical $(P_5,\text{ gem})$-free graphs in Section~\ref{sec:critP5gem}. In Section~\ref{sec:6and7} we give complete lists of all 19 graphs for $k=6$ and all 46 graphs for $k=7$ and outline our exhaustive computer search (the graphs in graph6 format and our code is available at \cite{graphfiles}). We then conclude by posing some open problems for future research in Section~\ref{sec:conclusion}.

\subsection{Definitions and Notation}

Let $\chi(G)$ denote the \textit{chromatic number} of $G$. A graph $G$ is \textit{$k$-vertex-critical} if $\chi(G)=k$ and $\chi(G-v)<k$ for all $v\in V(G)$. For a given colouring $c$ of a graph $G$ and subset $S\subseteq V(G)$, let $c(S)$ denote the set of all colours used on vertices in $S$. For $S\subseteq V(G)$, let $G[S]$ denote the subgraph of $G$ induced by $S$. For vertices $u,v\in V(G)$, we write $u\sim v$ if $u$ and $v$ are adjacent and $u\nsim v$ is $u$ and $v$ are non-adjacent. A for $S,T\subseteq V(G)$, we say $S$ is \textit{complete} (\textit{anti-complete}) to $T$ if $s\sim t$ ($s\nsim t$) for all $s\in S$ and $t\in T$. A \textit{module} of a graph $G$ is a set $S$ such that for all $v\in V(G)\setminus S$, $v$ is either complete or anti-complete to $S$. For $S\subseteq V(G)$, we say $S$ is a \textit{stable set} if $u\nsim v$ for all $u,v\in S$. A stable set is called \textit{good} if it meets every maximum clique of $G$. Given a graph $G$, a clique ($P_4$-free) expansion of $G$ is a graph obtained from $G$ by replacing some or all of its vertices by cliques ($P_4$-free graphs) and joining every vertex in these cliques ($P_4$-free graphs) according to the adjacencies of $G$.

\section{Vertex-critical ($P_5$, gem)-free graphs}\label{sec:critP5gem}
Throughout this section, we will make extensive use of a structural characterization give in~\cite{Chudetal2020} which requires some further definitions. Let $G_1,G_2,\ldots,G_{10}$ be defined as shown in Figure~\ref{fig:G1G7} and let $\mathcal{G}_i$ and $\mathcal{G}^{\ast}_i$ denote the set of all $P_4$-free and clique expansions of $G_i$, respectively. Throughout, for a graph $G\in \mathcal{G}_i$, we will use $ {Q_j}$ to denote the set of vertices used to replace $v_j$ in the graph $G_i$ according to the labellings in Figure~\ref{fig:G1G7}.
Let $\calH$ be the class of ($P_5$, gem)-free graphs $G$ such that $V(G)$ can be partitioned into seven non-empty sets, $A_1,\ldots,A_7$ such that:
\begin{itemize}
\item $A_1$ is complete to $A_2\cup A_5$ and anti-complete to $A_3\cup A_4\cup A_6\cup A_7$,
\item $A_2$ is complete to $A_1\cup A_3$ and anti-complete to $A_4\cup A_5\cup A_6\cup A_7$,
\item $A_3$ is complete to $A_2\cup A_4\cup A_6$ and anti-complete to $A_1\cup A_5\cup A_7$, 
\item $A_4$ is complete to $A_3\cup A_5\cup A_6$ and anti-complete to $A_1\cup A_2\cup A_7$,
\item $A_5$ is complete to $A_1\cup A_4$ and anti-complete to $A_2\cup A_3\cup A_6\cup A_7$,
\item every component of $G[A_7]$ is a module, and
\item every vertex in $A_7$ has a neighbour in $A_6$.
\end{itemize}

\noindent We refer the reader to Figure~\ref{fig:H} for the general form of a graph in $\calH$. Let $\Hstar\subseteq \calH$ be defined by all graphs in $\calH$ such that $A_i$ is a clique for all $i\in \{1,2,3,4,5\}$ and each component of $A_7$ is a clique. We are now ready for the structural characterization.
\setcounter{subfigure}{0}
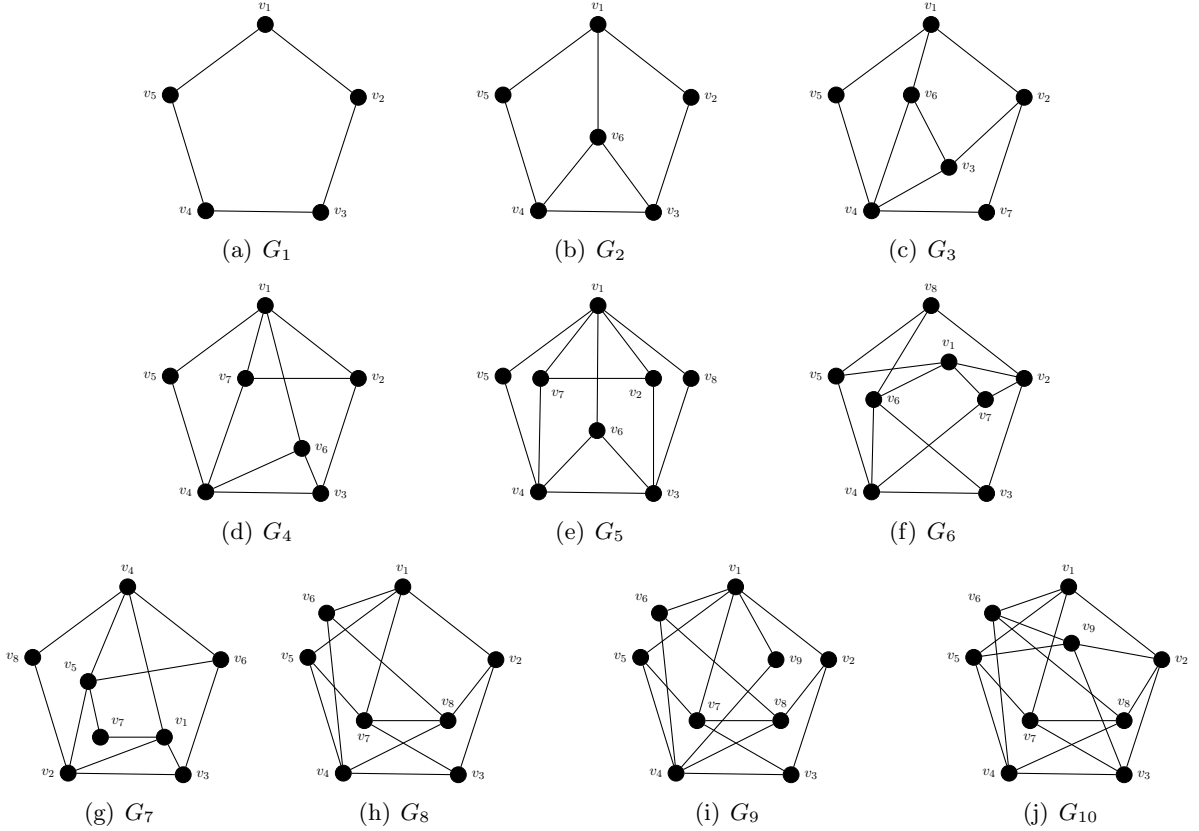
\begin{figure}[htb]
\def\c{0.5}
\def\r{1}
\centering
\subfigure[$G_1$]{
\scalebox{\c}{
\begin{tikzpicture}
\GraphInit[vstyle=Classic]
\Vertex[Lpos=180,L=\hbox{$v_5$},x=0.0cm,y=3.1235cm]{v0}
\Vertex[L=\hbox{$v_2$},x=5.0cm,y=3.0618cm]{v1}
\Vertex[Lpos=180,L=\hbox{$v_4$},x=0.942cm,y=0.0422cm]{v2}
\Vertex[Lpos=90,L=\hbox{$v_1$},x=2.5253cm,y=5.0cm]{v3}
\Vertex[L=\hbox{$v_3$},x=4.0006cm,y=0.0cm]{v4}

\Edge[](v0)(v2)
\Edge[](v0)(v3)
\Edge[](v1)(v3)
\Edge[](v1)(v4)
\Edge[](v2)(v4)
\end{tikzpicture}}}
\qquad
\subfigure[$G_2$]{
\scalebox{\c}{
\begin{tikzpicture}
\GraphInit[vstyle=Classic]
\Vertex[Lpos=180,L=\hbox{$v_5$},x=0.0cm,y=3.1235cm]{v4}
\Vertex[L=\hbox{$v_2$},x=5.0cm,y=3.0618cm]{v1}
\Vertex[Lpos=180,L=\hbox{$v_4$},x=0.942cm,y=0.0422cm]{v3}
\Vertex[Lpos=90,L=\hbox{$v_1$},x=2.5253cm,y=5.0cm]{v0}
\Vertex[L=\hbox{$v_3$},x=4.0006cm,y=0.0cm]{v2}
\Vertex[L=\hbox{$v_6$},x=2.5253cm,y=2cm]{v5}
\Edge[](v0)(v1)
\Edge[](v0)(v4)
\Edge[](v0)(v5)
\Edge[](v1)(v2)
\Edge[](v2)(v3)
\Edge[](v2)(v5)
\Edge[](v3)(v4)
\Edge[](v3)(v5)
\end{tikzpicture}}}
\qquad
\subfigure[$G_3$]{
\scalebox{\c}{
\begin{tikzpicture}
\GraphInit[vstyle=Classic]
\Vertex[Lpos=180,L=\hbox{$v_5$},x=0.0cm,y=3.1235cm]{v4}
\Vertex[L=\hbox{$v_2$},x=5.0cm,y=3.0618cm]{v1}
\Vertex[Lpos=180,L=\hbox{$v_4$},x=0.942cm,y=0.0422cm]{v3}
\Vertex[Lpos=90,L=\hbox{$v_1$},x=2.5253cm,y=5.0cm]{v0}
\Vertex[L=\hbox{$v_3$},x=3cm,y=1.2cm]{v2}
\Vertex[L=\hbox{$v_6$},x=2cm,y=3.1235cm]{v5}
\Vertex[L=\hbox{$v_7$},x=4.0006cm,y=0.0cm]{v6}
\Edge[](v0)(v1)
\Edge[](v0)(v4)
\Edge[](v0)(v5)
\Edge[](v1)(v2)
\Edge[](v1)(v6)
\Edge[](v2)(v3)
\Edge[](v2)(v5)
\Edge[](v3)(v4)
\Edge[](v3)(v5)
\Edge[](v3)(v6)
\end{tikzpicture}}}
\qquad
\subfigure[$G_4$]{
\scalebox{\c}{
\begin{tikzpicture}
\GraphInit[vstyle=Classic]
\Vertex[Lpos=180,L=\hbox{$v_5$},x=0.0cm,y=3.1235cm]{v4}
\Vertex[L=\hbox{$v_2$},x=5.0cm,y=3.0618cm]{v1}
\Vertex[Lpos=180,L=\hbox{$v_4$},x=0.942cm,y=0.0422cm]{v3}
\Vertex[Lpos=90,L=\hbox{$v_1$},x=2.5253cm,y=5.0cm]{v0}
\Vertex[L=\hbox{$v_3$},x=4.0006cm,y=0.0cm]{v2}
\Vertex[L=\hbox{$v_6$},x=3.5cm,y=1.2cm]{v5}
\Vertex[Lpos=180,L=\hbox{$v_7$},x=2cm,y=3.0618cm]{v6}
\Edge[](v0)(v1)
\Edge[](v0)(v4)
\Edge[](v0)(v5)
\Edge[](v0)(v6)
\Edge[](v1)(v2)
\Edge[](v1)(v6)
\Edge[](v2)(v3)
\Edge[](v2)(v5)
\Edge[](v3)(v4)
\Edge[](v3)(v5)
\Edge[](v3)(v6)
\end{tikzpicture}}}
\qquad
\subfigure[$G_5$]{
\scalebox{\c}{
\begin{tikzpicture}
\GraphInit[vstyle=Classic]
\Vertex[Lpos=180,L=\hbox{$v_5$},x=0.0cm,y=3.1235cm]{v4}
\Vertex[Lpos=-135,L=\hbox{$v_2$},x=4cm,y=3.0618cm]{v1}
\Vertex[Lpos=180,L=\hbox{$v_4$},x=0.942cm,y=0.0422cm]{v3}
\Vertex[Lpos=90,L=\hbox{$v_1$},x=2.5253cm,y=5.0cm]{v0}
\Vertex[L=\hbox{$v_3$},x=4.0006cm,y=0.0cm]{v2}
\Vertex[L=\hbox{$v_8$},x=5.0cm,y=3.0618cm]{v7}
\Vertex[L=\hbox{$v_6$},x=2.5cm,y=1.6729cm]{v5}
\Vertex[Lpos=-45,L=\hbox{$v_7$},x=1cm,y=3.0618cm]{v6}
\Edge[](v0)(v1)
\Edge[](v0)(v4)
\Edge[](v0)(v5)
\Edge[](v0)(v6)
\Edge[](v0)(v7)
\Edge[](v1)(v2)
\Edge[](v1)(v6)
\Edge[](v2)(v3)
\Edge[](v2)(v5)
\Edge[](v2)(v7)
\Edge[](v3)(v4)
\Edge[](v3)(v5)
\Edge[](v3)(v6)
\end{tikzpicture}}}
\qquad
\subfigure[$G_6$]{
\scalebox{\c}{
\begin{tikzpicture}
\GraphInit[vstyle=Classic]
\Vertex[Lpos=180,L=\hbox{$v_5$},x=0.0cm,y=3.1235cm]{v4}
\Vertex[L=\hbox{$v_2$},x=5.0cm,y=3.0618cm]{v1}
\Vertex[Lpos=180,L=\hbox{$v_4$},x=0.942cm,y=0.0422cm]{v3}
\Vertex[Lpos=90,L=\hbox{$v_1$},x=3cm,y=3.5cm]{v0}
\Vertex[L=\hbox{$v_3$},x=4.0006cm,y=0.0cm]{v2}
\Vertex[Lpos=90,L=\hbox{$v_8$},x=2.5253cm,y=5.0cm]{v7}
\Vertex[L=\hbox{$v_6$},x=1cm,y=2.5cm]{v5}
\Vertex[Lpos=-90,L=\hbox{$v_7$},x=3.9687cm,y=2.5cm]{v6}
\Edge[](v0)(v1)
\Edge[](v0)(v4)
\Edge[](v0)(v5)
\Edge[](v0)(v6)
\Edge[](v1)(v2)
\Edge[](v1)(v6)
\Edge[](v1)(v7)
\Edge[](v2)(v3)
\Edge[](v2)(v5)
\Edge[](v3)(v4)
\Edge[](v3)(v5)
\Edge[](v3)(v6)
\Edge[](v4)(v7)
\Edge[](v5)(v7)
\end{tikzpicture}}}
\qquad
\subfigure[$G_7$]{
\scalebox{\c}{
\begin{tikzpicture}
\GraphInit[vstyle=Classic]
\Vertex[Lpos=45,L=\hbox{$v_1$},x=3.5cm,y=1cm]{v0}
\Vertex[Lpos=180,L=\hbox{$v_2$},x=0.942cm,y=0.0422cm]{v1}
\Vertex[L=\hbox{$v_3$},x=4.0006cm,y=0.0cm]{v2}
\Vertex[Lpos=90,L=\hbox{$v_4$},x=2.5253cm,y=5.0cm]{v3}
\Vertex[Lpos=135,L=\hbox{$v_5$},x=1.4765cm,y=2.4803cm]{v4}
\Vertex[L=\hbox{$v_6$},x=5.0cm,y=3.0618cm]{v5}
\Vertex[Lpos=45,L=\hbox{$v_7$},x=1.8cm,y=1cm]{v6}
\Vertex[Lpos=180,L=\hbox{$v_8$},x=0.0cm,y=3.1235cm]{v7}
\Edge[](v0)(v1)
\Edge[](v0)(v2)
\Edge[](v0)(v3)
\Edge[](v0)(v6)
\Edge[](v1)(v2)
\Edge[](v1)(v4)
\Edge[](v1)(v7)
\Edge[](v2)(v5)
\Edge[](v3)(v4)
\Edge[](v3)(v5)
\Edge[](v3)(v7)
\Edge[](v4)(v5)
\Edge[](v4)(v6)
\end{tikzpicture}}}
\subfigure[$G_8$]{
\scalebox{\c}{
\begin{tikzpicture}
\GraphInit[vstyle=Classic]
\Vertex[Lpos=180,L=\hbox{$v_5$},x=0.0cm,y=3.1235cm]{v4}
\Vertex[L=\hbox{$v_2$},x=5.0cm,y=3.0618cm]{v1}
\Vertex[Lpos=180,L=\hbox{$v_4$},x=0.942cm,y=0.0422cm]{v3}
\Vertex[Lpos=90,L=\hbox{$v_1$},x=2.5253cm,y=5.0cm]{v0}
\Vertex[L=\hbox{$v_3$},x=4.0006cm,y=0.0cm]{v2}

\Vertex[Lpos=135,L=\hbox{$v_6$},x=0.5cm,y=4.3cm]{v5}
\Vertex[Lpos=-90,L=\hbox{$v_7$},x=1.5cm,y=1.4398cm]{v6}
\Vertex[Lpos=90,L=\hbox{$v_8$},x=3.7265cm,y=1.4398cm]{v7}
\Edge[](v0)(v1)
\Edge[](v0)(v4)
\Edge[](v0)(v5)
\Edge[](v0)(v6)
\Edge[](v1)(v2)
\Edge[](v1)(v7)
\Edge[](v2)(v3)
\Edge[](v2)(v6)
\Edge[](v3)(v4)
\Edge[](v3)(v5)
\Edge[](v3)(v7)
\Edge[](v4)(v6)
\Edge[](v5)(v7)
\Edge[](v6)(v7)
\end{tikzpicture}}}
\qquad
\subfigure[$G_9$]{
\scalebox{\c}{
\begin{tikzpicture}
\GraphInit[vstyle=Classic]
\Vertex[Lpos=180,L=\hbox{$v_5$},x=0.0cm,y=3.1235cm]{v4}
\Vertex[L=\hbox{$v_2$},x=5.0cm,y=3.0618cm]{v1}
\Vertex[Lpos=180,L=\hbox{$v_4$},x=0.942cm,y=0.0422cm]{v3}
\Vertex[Lpos=90,L=\hbox{$v_1$},x=2.5253cm,y=5.0cm]{v0}
\Vertex[L=\hbox{$v_3$},x=4.0006cm,y=0.0cm]{v2}

\Vertex[Lpos=135,L=\hbox{$v_6$},x=0.5cm,y=4.3cm]{v5}
\Vertex[Lpos=45,L=\hbox{$v_7$},x=1.5cm,y=1.4398cm]{v6}
\Vertex[Lpos=90,L=\hbox{$v_8$},x=3.7265cm,y=1.4398cm]{v7}
\Vertex[L=\hbox{$v_9$},x=3.6cm,y=3.0618cm]{v8}
\Edge[](v0)(v1)
\Edge[](v0)(v4)
\Edge[](v0)(v5)
\Edge[](v0)(v6)
\Edge[](v0)(v8)
\Edge[](v1)(v2)
\Edge[](v1)(v7)
\Edge[](v2)(v3)
\Edge[](v2)(v6)
\Edge[](v3)(v4)
\Edge[](v3)(v5)
\Edge[](v3)(v7)
\Edge[](v3)(v8)
\Edge[](v4)(v6)
\Edge[](v5)(v7)
\Edge[](v6)(v7)
\end{tikzpicture}}}
\qquad
\subfigure[$G_{10}$]{
\scalebox{\c}{
\begin{tikzpicture}
\GraphInit[vstyle=Classic]
\Vertex[Lpos=180,L=\hbox{$v_5$},x=0.0cm,y=3.1235cm]{v4}
\Vertex[L=\hbox{$v_2$},x=5.0cm,y=3.0618cm]{v1}
\Vertex[Lpos=180,L=\hbox{$v_4$},x=0.942cm,y=0.0422cm]{v3}
\Vertex[Lpos=90,L=\hbox{$v_1$},x=2.5253cm,y=5.0cm]{v0}
\Vertex[L=\hbox{$v_3$},x=4.0006cm,y=0.0cm]{v2}

\Vertex[Lpos=135,L=\hbox{$v_6$},x=0.5cm,y=4.3cm]{v5}
\Vertex[Lpos=-90,L=\hbox{$v_7$},x=1.5cm,y=1.4398cm]{v6}
\Vertex[Lpos=90,L=\hbox{$v_8$},x=4cm,y=1.4398cm]{v7}
\Vertex[Lpos=45,L=\hbox{$v_9$},x=2.6cm,y=3.5cm]{v8}
\Edge[](v0)(v1)
\Edge[](v0)(v4)
\Edge[](v0)(v5)
\Edge[](v0)(v6)
\Edge[](v1)(v2)
\Edge[](v1)(v7)
\Edge[](v1)(v8)
\Edge[](v2)(v3)
\Edge[](v2)(v6)
\Edge[](v2)(v8)
\Edge[](v3)(v4)
\Edge[](v3)(v5)
\Edge[](v3)(v7)
\Edge[](v4)(v6)
\Edge[](v4)(v8)
\Edge[](v5)(v7)
\Edge[](v5)(v8)
\Edge[](v6)(v7)
\end{tikzpicture}}}
\caption{Special $(P_5,\text{ gem})$-free graphs used in the structural characterization in \cite{Chudetal2020}.}\label{fig:G1G7}
\end{figure}

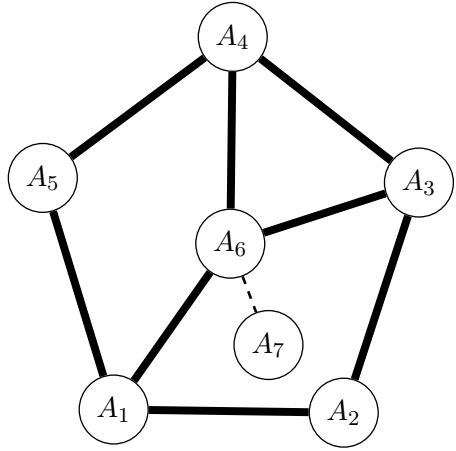
\begin{figure}[htb]
\centering
\begin{tikzpicture}
\GraphInit[vstyle=Normal]
\Vertex[L=\hbox{$A_5$},x=0.0cm,y=3.1235cm]{A1}
\Vertex[L=\hbox{$A_3$},x=5.0cm,y=3.0618cm]{A3}
\Vertex[L=\hbox{$A_1$},x=0.942cm,y=0.0422cm]{A5}
\Vertex[L=\hbox{$A_4$},x=2.5253cm,y=5.0cm]{A2}
\Vertex[L=\hbox{$A_2$},x=4.0006cm,y=0.0cm]{A4}
\Vertex[L=\hbox{$A_6$},x=2.4923cm,y=2.2487cm]{A6}
\Vertex[L=\hbox{$A_7$},x=3cm,y=.9cm]{A7}
\SetUpEdge[lw=3pt]

\Edge[](A1)(A2)
\Edge[](A1)(A5)
\Edge[](A2)(A3)
\Edge[](A2)(A6)
\Edge[](A3)(A4)
\Edge[](A3)(A6)
\Edge[](A4)(A5)
\Edge[](A6)(A5)
\SetUpEdge[lw=1pt]
\Edge[style=dashed](A6)(A7)
%
\end{tikzpicture}
\caption{General form of a graph in $\calH$ where the thick black lines denote sets which are complete to each other and the thin dashed line denotes any adjacency satisfying the requirements in~\cite{Chudetal2020}.}\label{fig:H}
\end{figure}

\begin{theorem}[\cite{Chudetal2020}]\label{thm:P5gemfreestructure}
If $G$ is a $(P_5,\text{ gem})$-free graph, then $G$ is either perfect, in $\mathcal{G}_i$ for some $i\in\{1,\ldots, 10\}$, or $G\in\calH$.
\end{theorem}

Since we are interested in vertex-critical graphs, the following result will be helpful.
\begin{lemma}[\cite{Dhaliwal2017}]\label{lem:module}
If $G$ is $k$-vertex-critical with a non-trivial module $M$, then $M$ is $m$-vertex-critical for some $m<k$.
\end{lemma}

\begin{lemma}\label{lem:kcritinGastorHast}
If $G$ is a $k$-vertex-critical $(P_5,\text{ gem})$-free graph, then $G$ is $K_k$, $G\in \Hstar$, or $G\in \mathcal{G}_i^{\ast}$ for some $i=1,\ldots 10$.
\end{lemma}
\begin{proof}
Let $G$ be a $k$-vertex-critical $(P_5,\text{ gem})$-free graph. If $G$ is perfect, then $G=K_k$, otherwise $G\in\mathcal{G}_i$ for some $i\in\{1,\ldots, 10\}$, or $G\in\calH$ from Theorem~\ref{thm:P5gemfreestructure}. Suppose $G\in\mathcal{G}_i$ for some $i\in\{1,\ldots, 10\}$. By definition, every $Q_j$ of $\mathcal{G}_i$ is a non-trivial module and therefore must be $k$-vertex-critical from Lemma~\ref{lem:module}. Further, since each $Q_j$ must induce a $P_4$-free graph, it follows that each must be a clique. Therefore, $G\in\mathcal{G}_i^{\ast}$. Similarly, if $G\in \calH$, then $A_i$ is a non-trivial module that induces a $P_4$-free graph for all $i\in \{1,2,3,4,5\}$. Thus, by Lemma~\ref{lem:module}, $A_i$ is a clique for all $i\in \{1,2,3,4,5\}$ and therefore $G\in\Hstar$.
\end{proof}

With Lemma~\ref{lem:kcritinGastorHast} in hand, the most difficult case to deal with is $k$-vertex-critical graphs in $\Hstar$. We will show that there are no such graphs in $\Hstar$, but it will require several lemmas first to help refine our understanding of the structure.

\begin{lemma}\label{lem:C5sinH}
If $G\in \Hstar$ and $S\subset V(G)$ such that $G[S]\cong C_5$, then $S\subseteq \cup_{i=1}^{5}A_i$ or $S\subseteq A_6\cup A_7$.
\end{lemma}
\begin{proof}
Let $G\in {\Hstar}$ and $S\subset V(G)$ such that $G[S]\cong C_5$. Clearly, if $S\cap A_7=\emptyset$, then $S\subseteq \cup_{i=1}^{5}A_i$. So suppose $S\cap A_7\neq \emptyset$ and let $a_7\in S\cap A_7$. Since $A_7$ is the disjoint union of cliques 
and each component is a module whose neighbours are contained in $A_6$, we must have a vertex $a_6\in A_6 \cap S$. If there is a vertex $a_7'\in S\cap A_7$ such that $a_7'\sim a_7$, then $\{a_6,a_7',a_7\}$ induces a triangle, contradicting $G[S]\cong C_5$. Since $a_7$ must have another neighbour, we must have $a_6'\in A_6$ such that $a_6\neq a_6'$, $a_6'\nsim a_6$, and $a_6'\sim a_7$. Now if $S\not\subseteq A_6\cup A_7$, then there must be a vertex $s\in S\cap (A_1\cup A_3\cup A_4)$. However, now $\{s,a_6,a_6',a_7\}$ induces a $C_4$ in $G$, contradicting $G[S]\cong C_5$. Therefore, $S\subseteq A_6\cup A_7$.
\end{proof}

\begin{lemma}\label{lem:HisG4G10free}
For all $G\in \Hstar$, $G$ is $(G_4,G_5,G_6,G_7,G_8,G_9,G_{10})$-free. 
\end{lemma}
\begin{proof}

By way of contradiction, suppose $G\in {\Hstar}$ and $S\subset V(G)$ such that $G[S]\cong G_i$ for some $i\in \{4,5,\ldots, 10\}$. Thus, there is a $D\subset S$ such that $D$ induces a $C_5$ and it follows from Lemma~\ref{lem:C5sinH} that $D\subseteq \cup_{i=1}^{5}A_i$ or $D\subseteq A_6\cup A_7$. 

Suppose that $D\subseteq \cup_{i=1}^{5}A_i$. Then each of the five vertices of $D$ belongs to a distinct $A_i$ for $i = 1,2, \ldots, 5$, since each of these $A_i$'s induces a clique. From the structure of graphs in $\Hstar$, every neighbour of a vertex in $D$ must be in $A_6$. Since $A_6$ is a module in $G-A_7$, it must be the case that $N(u)\cap D=N(v)\cap D$ for all $u,v\in S\setminus D$ such that $N(u)\cap D\neq\emptyset$ and $N(v)\cap D\neq\emptyset$.
However, we find that this is not the case for any induced $C_5$ of $G_i$ for any $i\in\{4,5,\ldots,10\}$. For example, for $G_4$, $C=\{v_1,v_2,v_3,v_4,v_5\}$ induces the only $C_5$ in the graph and $N(v_6)\cap C=\{v_1,v_3,v_4\}\neq N(v_7)\cap C=\{v_1,v_2,v_4\}$.  Therefore, $S$ cannot induce $G_i$ for $i = 4, \dots 10$, a contradiction. Thus, $D\not\subseteq \cup_{i=1}^{5}A_i$.

Thus, we may assume that $D\subseteq A_6\cup A_7$. Since each component of $A_7$ is a module and a clique, if $x,y\in D \cap A_7$ such that $x\sim y$, then $N[x]=N[y]$, otherwise a triangle is induced. But  for all $j\in \{4,5,\ldots, 10\}$ and $u,v\in G_j$ such that $u\sim v$, $N[u]\neq N[v]$. Therefore, $D\cap A_7$ induces a stable set in $G$. For any stable set $T$ of $G_j$, it is easily verifiable that $G_j-T$ contains an induced $P_4$. Therefore, $S\setminus A_7$ must contain a subset $P$ such that $G[P]\cong P_4$. Note that $A_6$ is $P_4$-free since $G\in\Hstar$, so $P\not\subseteq A_6$ and therefore $S\not\subseteq A_7\cup A_6$. Further, $|P\cap A_7|=1$ since $A_6$ is complete to $A_1\cup A_3\cup A_4$. Thus $P$ must either follow the path $A_5-A_1-A_6-A_3$ or $A_2-A_1-A_6-A_4$. In either case, $P$ starts at a vertex with no neighbours in $D$ (i.e., an induced $C_5$) and ends at a vertex not in $D$ but with a neighbour in $D$. Thus, $S$ must induce $G_7$ as it can readily be checked that this is the only graph with vertices that have no neighbours on an induced $C_5$. In fact, $G_7$ has exactly two such vertices: $v_7$ has no neighbours on the $C_5$ induced by $C=\{v_2,v_3,v_6,v_4,v_8\}$ and $v_8$ has no neighbours on the $C_5$ induced by $C'=\{v_1,v_3,v_6,v_5,v_7\}$. However, all induced $P_4$'s in $G_7$ that start with $v_7$ end on a vertex in the set $\{v_6, v_8, v_3\}\subseteq C$, and all that start with $v_8$ end on a vertex in the set $\{v_6, v_7, v_3\}\subseteq C'$. This contradicts the existence of $P$. Therefore, $D\not\subseteq A_6\cup A_7$ which completes the proof.
\end{proof}


\begin{lemma}[Lemma 3.1 in \cite{Hoang2015}]\label{lem:similar}
	Let $G$ be a graph with two disjoint cliques $A$ and $B$ with $A= a_1, a_2, \ldots, a_m$ and $B = b_1, b_2 , \ldots , b_m$ such that $N(a_i) \setminus A \subseteq N(b_i) \setminus B$ for all $1 \leq i \leq m$. Then $G$ is not $k$-vertex-critical.
\end{lemma}


While we are still focused on showing there are no vertex-critical graphs in $\Hstar$, the following lemma will be required.

\begin{lemma}\label{lem:nocritinG2G3}
There are no vertex-critical graphs in $\mathcal{G}_2\cup \mathcal{G}_3$.
\end{lemma}
\begin{proof}
Suppose $G$ is a $k$-vertex-critical graph in $\mathcal{G}_2\cup\mathcal{G}_3$ for some $k$  {and recall that $Q_i$ denotes the set of vertices replacing $v_i$ as labelled in Figure~\ref{fig:G1G7}}. By Lemma~\ref{lem:module}, it follows that $G\in \mathcal{G}_2^{\ast}\cup\mathcal{G}^{\ast}_3$. If $|Q_6| \geq |Q_5|$, then $Q_6$ contains a clique $Q'_6$ with $|Q'_6| = |Q_5|$ vertices such that $Q'_6$ and $Q_5$ that satisfy the hypothesis of Lemma~\ref{lem:similar}, so $G$ is not $k$-vertex-critical, a contradiction. So, we have $|Q_6|<|Q_5|$.

Fix a $k$-colouring of $c:V(G)\to\{1,\ldots,k\}$ of $G$ such that $x\in Q_6$ is the only vertex with colour $k$. Since $\chi(G)=k$, all colours in $\{1,\ldots,k-1\}$ appear in the neighbourhood of $x$, so $c(Q_5)\subseteq c(Q_6)\cup c(Q_3)$. Note that $|c(Q_6)|=|Q_6|$ since $Q_6$ is a clique. 

Since $|Q_6|< |Q_5|$, it follows that there is a colour $c_5\in c(Q_5)\cap c(Q_3)$. Let $x_5\in Q_5$ be such that $c(x_5)=c_5$. Further, if $c(Q_1)\subseteq c(Q_4)\cup c(Q_3)$, then $k=|Q_6|+|Q_4|+|Q_3|=\chi(G[Q_3\cup Q_4\cup Q_6])$, contradicting $G$ being $k$-vertex-critical. Therefore, there is a colour $c_1\in c(Q_1)\setminus (c(Q_4)\cup c(Q_3))$. Let $x_1\in Q_1$ be such that $c(x_1)=c_1$. Thus, $c':V(G)\to\{1,2,\ldots,k\}$ defined by $c'(v)=c(v)$ for all $v\in V(G)\setminus\{x_1,x_5\}$ and $c'(x_5)=c_1$ and $c'(x_1)=c_5$ is a $k$-colouring of $G$. Finally, we can now change the colour of $x$ to $c_1$ to obtain a $(k-1)$-colouring of $G$, a contradiction. Therefore $\chi(G-x)=\chi(G)$ for all $x\in Q_6$, and since $Q_6\neq \emptyset$, this contradicts $G$ being $k$-vertex-critical. Therefore, there are no $k$-vertex-critical graphs in $\mathcal{G}_2\cup\mathcal{G}_3$.
\end{proof}

\begin{lemma}[\cite{KarthickMaffraygemcogem}]\label{lem:chromaticnumberofC5exp}
If $G\in \mathcal{G}_{1}^{\ast}$ has order $n$, then $\chi(G)=\max\left(\omega(G),\left\lceil\frac{n}{2}\right\rceil\right)$.
\end{lemma}

\begin{theorem}\label{thm:nocritinH}
There are no critical graphs in $\calH$.
\end{theorem}
\begin{proof}
By induction on $k$.  For $k=1,2,3$ the result is clear. For some $k>3$, suppose there are no $k'$-vertex-critical graphs in $\calH$ for all $k'<k$ and suppose $G$ is a $k$-vertex-critical in $\calH$. By Lemma~\ref{lem:module}, it follows that $G\in \Hstar$. Let $T_1,T_2,\ldots,T_m$ be the components of $A_7$, $C=\cup_{i=1}^{5}A_i$, and $q=\omega(G)$. Note that $q<k$ since $G$ is $k$-vertex-critical and not complete.
From Lemma~\ref{lem:similar}, it follows that 
\begin{eqnarray}
\omega(A_6)&<&|A_2|\label{eq:A6<A2}\\
\omega(A_6)&<&|A_5|.\label{eq:A6<A5}
\end{eqnarray}

\noindent It follows from (\ref{eq:A6<A2}) and (\ref{eq:A6<A5}) that no clique in $A_1\cup A_6$ is a $q$-clique.  {We now consider two cases}.\\




\noindent\textit{{Case 1:}} {For some $c_1,c_2\in C$ and $t_i\in T_i$, $S=\{c_1,c_2,t_1,t_2,\ldots, t_m\}$ is a good stable set in $G$.}\\

Then $G-S$ must be $(k-1)$-chromatic and $\omega(G-S)<q\le k-1$. If $G-S$ does not have an induced $C_5$, then since it is $(P_5,\text{ gem})$-free, it must be perfect by the Strong Perfect Graph Theorem~\cite{Chudnovsky2006} and therefore must contain $K_{k-1}$, contradicting $\omega(G-S)<k-1$.  Thus we may assume $G - S$ contains an induced $C_5$.  Since there are no $(k-1)$-vertex-critical graphs in $\calH$ by the inductive hypothesis, $G-S$ must have a $(k-1)$-vertex-critical induced subgraph $F$ that is in $\mathcal{G}_i^{\ast}$ for some $i\in \{1,\ldots,10\}$ by Theorem~\ref{thm:P5gemfreestructure}.  By Lemma~\ref{lem:HisG4G10free}, we have $F\in \mathcal{G}_{1}^{\ast}\cup \mathcal{G}_{2}^{\ast}\cup \mathcal{G}_{3}^{\ast}$. From Lemma~\ref{lem:nocritinG2G3}, we must have $F\in \mathcal{G}_{1}^{\ast}$. Since $\omega(G-S)<k-1$, Lemma~\ref{lem:chromaticnumberofC5exp} gives $k-1=\left\lceil\frac{|V(F)|}{2}\right\rceil$. By Lemma~\ref{lem:C5sinH}, we must have $V(F)\subseteq C\setminus\{c_1,c_2\}$ or $V(F)\subseteq A_6\cup A_7\setminus\{t_1,t_2\ldots, t_m\}$.

If $V(F)\subseteq C\setminus\{c_1,c_2\}$, then $G[V(F)\cup\{c_1,c_2\}]\in \mathcal{G}_1^{\ast}$ and has order $|V(F)|+2$. Therefore, by Lemma~\ref{lem:chromaticnumberofC5exp}, we have $\chi(G[V(F)\cup\{c_1,c_2\}])=k$, which contradicts $G$ being $k$-vertex-critical. So, we may assume  $V(F)\subseteq A_6\cup A_7\setminus\{t_1,t_2\ldots, t_m\}$. If there is at most one $i\in\{1,\ldots,m\}$ such that $T_i\cap V(F)\neq\emptyset$, then $A_6\cap V(F)$ must induce a graph containing a $P_4$. But then $G[A_6\cup\{x\}]$ contains an induced gem for all $x\in A_1\cup A_3\cup A_4$, contradicting $G$ being gem-free. So, for some distinct $i,j\in \{1,\ldots, m\}$, we have $T_i\cap V(F)\neq \emptyset$ and $T_j\cap V(F)\neq\emptyset$. Therefore, $G[V(F)\cup\{t_i,t_j\}]\in \mathcal{G}_1^{\ast}$ and has order $|V(F)|+2$. Hence $\chi(G[V(F)\cup\{t_i,t_j\}])=k$ by Lemma~\ref{lem:chromaticnumberofC5exp}, which contradicts $G$ being $k$-vertex-critical.\\

\noindent\textit{ {Case 2:}}  {For all $c_1,c_2\in C$ and $t_i\in T_i$, $S=\{c_1,c_2,t_1,t_2,\ldots, t_m\}$ is not a good stable set in $G$.}\\

Then since no clique in $A_1\cup A_6$ is a $q$-clique, from the proof of \cite[Theorem 8]{Chudetal2020}, we have $|A_1|=|A_3|=|A_4|=a$ and $|A_2|=|A_5|=q-a$ (for example, if $A_1\cup A_2$ is not a $q$-clique, then the set $S$ defined in the next sentence is a good stable set). Let $S=\{x_3,x_5,t_1,t_2,\ldots, t_m\}$ for some $t_i\in T_i$, $x_3\in A_3$, and $x_5\in A_5$. If $q<k-1$, then since $G$ is $k$-vertex-critical and $S$ is a stable set, we have $\chi(G-S)=k-1$ and have $\omega(G-S)\le q< k-1$ and similar to above, $G-S$ must have a $(k-1)$-vertex-critical induced subgraph $F$, such that $F\in \mathcal{G}_1^{\ast}$ and either $G[F\cup\{x_3,x_5\}]$ or $G[F\cup \{t_i,t_j\}]$ are $k$-chromatic proper subgraphs of $G$ for some $i,j\in\{1,\ldots,m\}$, contradicting $G$ being $k$-vertex-critical.

So suppose $q=k-1$. If $a=1$, then $G-(A_6\cup A_7)$ is a $k$-vertex-critical graph, a contradiction. So suppose $a>1$. Since $G$ is $k$-vertex-critical, there is a  $k$-colouring $G$ such that $x_5\in A_5$ is the only vertex with colour $k$ and $k-1$ colours appear in the neighbourhood of $x_5$. Without loss of generality, we may assume  colours $a+2,a+3,\ldots, k-1,k$ appear in $A_5$. Since $A_5$ is complete to $A_4$, the $k$-colouring must use $a$ new colours on $A_4$, so we may assume that colours $2,3,\ldots, a,a+1$ appear in $A_4$. It follows that the vertices of $A_1$ are coloured with colours $1,2,\ldots, a$. Since $A_2$ is complete to $A_1$ and $x_5$ is the only vertex coloured $k$, the $(q-a)$-clique $A_2$ is coloured with colours $a+1,a+2,\ldots,k-1$. But now, since $A_3$ is complete to both $A_4$ and $A_2$ and  colour $k$ does not appear in $A_3$, colour $1$ is the only colour that can be used to colour the $a$-clique $A_3$. However, $a>1$, so no such $k$-colouring exists.   This contradicts $G$ being $k$-vertex-critical. 

Therefore, there are no critical graphs in $\calH$.
\end{proof}

\begin{theorem}\label{thm:finitemanycritp5gemfree}
There are only finitely many $k$-vertex-critical $(P_5,\text{ gem})$-free graphs for all $k$.
\end{theorem}
\begin{proof}
Let $G$ be a $k$-vertex-critical $(P_5,\text{ gem})$-free graph. By Lemma~\ref{lem:kcritinGastorHast}, we know that $G=K_k$, $G\in \mathcal{G}_i^{\ast}$, or $G\in\Hstar$. From Theorem~\ref{thm:nocritinH}, we know that $G\not\in\Hstar$. If $G\in\mathcal{G}_i^{\ast}$, then as in~\cite{CaiGoedgebeurHuang2021}, each $Q_j$ can only be a clique of order at most $k-2$ (otherwise, since there are no isolated vertices in any of the $G_i$'s, a $k$-clique will be a proper subgraph of $G$) $|V(G)|\le (k-2)^{|V(G_i|}\le (k-2)^{9}$, and therefore there are only finitely many such graphs.  
\end{proof}

\section{Complete Classification for $k=6$ and $k=7$}\label{sec:6and7}
As outlined in the proof of Theorem~\ref{thm:finitemanycritp5gemfree}, every $k$-vertex-critical ($P_5$, gem)-free graph is $K_k$ or a clique expansion of $G_i$ for some $i\in \{1,4,5,6,7,8,9,10\}$. This allowed us to implement an exhaustive computer search for  all such graphs for $k\le 7$ in Sagemath. Our search yielded exactly $19$ $6$-vertex-critical $(P_5,\text{ gem})$-free graphs and $46$ $7$-vertex-critical $(P_5,\text{ gem})$-free graphs. 
The code used to generate these graphs as well as the graphs in graph6 format are available at~\cite{graphfiles}. The idea behind our generation is quite simple: for each $G_i$ for $i\in\{1,4,5,6,7,8,9,10\}$ we construct all possible nonisomorphic clique expansions with clique number at most $k-1$ and perform a brute-force test to check if each graph is $k$-vertex-critical or not. To generate all nonisomorphic clique expansions with sufficiently low clique number, we first enumerate all possible $|V(G_i)|$-tuples where each element is between $1$ and $k-2$ and these will correspond to the replacing vertex $v_j$ of $G_i$ with a clique whose order is equal to the $j$-th element in the tuple. The tuples ones that would create cliques that are too large or would produce an isomorphic graph to a previous $|V(G_i)|$-tuple are removed from the list. It should be noted that both tuple generation and the vertex-critical test are computationally expensive\footnote{The computation for $k=7$ took approximately 72 hours on a machine with an AMD Ryzen 7 4800H Processor and 16GB of RAM.}. The complete lists are included below. For these lists, let $G_1(n_1,n_2,n_3,n_5)$ ($G_7(n_1,n_2,n_3,n_5,n_6,n_7)$) denote the clique expansion of $G_1$ ($G_7$) with vertex $v_i$ replaced by a clique of order $n_i$.\\

The $19$ $6$-vertex-critical $(P_5,\text{ gem})$-free graphs are given in the following list:\\
\tiny
\begin{multicols}{2}
\begin{enumerate}[nosep]
 \item $K_6$
 \item $G_1(4, 1, 4, 1, 1)$
 \item $G_1(4, 1, 3, 2, 1)$
 \item $G_1(3, 2, 3, 2, 1)$
 \item $G_1(3, 2, 2, 3, 1)$
 \item $G_1(3, 2, 2, 2, 2)$
 \item $G_7(1, 3, 1, 3, 1, 1, 4, 2)$
 \item $G_7(1, 2, 2, 3, 1, 1, 4, 2)$
 \item $G_7(1, 1, 3, 3, 1, 1, 4, 2)$
 \item $G_7(1, 2, 2, 2, 1, 2, 4, 2)$
 \item $G_7(2, 2, 1, 2, 2, 1, 3, 3)$
 \item $G_7(2, 1, 2, 2, 2, 1, 3, 3)$
 \item $G_7(1, 1, 3, 2, 2, 1, 3, 3)$
 \item $G_7(2, 1, 2, 2, 1, 2, 3, 3)$
 \item $G_7(1, 2, 2, 2, 1, 1, 4, 3)$
 \item $G_7(1, 1, 3, 2, 1, 1, 4, 3)$
 \item $G_7(1, 2, 2, 1, 1, 2, 4, 3)$
 \item $G_7(1, 1, 3, 1, 1, 1, 4, 4)$
 \item $G_7(1, 1, 2, 1, 1, 2, 4, 4)$
\end{enumerate}
\end{multicols}

\normalsize
\vspace{5mm}
The $46$ $7$-vertex-critical $(P_5,\text{ gem})$-free graphs are given in the following list:\\
\tiny
\begin{multicols}{3}
\begin{enumerate}[nosep]
 \item $K_7$
 \item $G_1(5, 1, 5, 1, 1)$
 \item $G_1(5, 1, 4, 2, 1)$
 \item $G_1(4, 2, 4, 2, 1)$
 \item $G_1(5, 1, 3, 3, 1)$
 \item $G_1(4, 2, 3, 3, 1)$
 \item $G_1(3, 3, 3, 3, 1)$
 \item $G_1(4, 2, 2, 4, 1)$
 \item $G_1(4, 2, 3, 2, 2)$
 \item $G_1(3, 3, 3, 2, 2)$
 \item $G_1(3, 3, 2, 3, 2)$
 \item $G_7(1, 4, 1, 4, 1, 1, 5, 2)$
 \item $G_7(1, 3, 2, 4, 1, 1, 5, 2)$
 \item $G_7(1, 2, 3, 4, 1, 1, 5, 2)$
 \item $G_7(1, 1, 4, 4, 1, 1, 5, 2)$
 \item $G_7(1, 3, 2, 3, 1, 2, 5, 2)$
 \item $G_7(1, 2, 3, 3, 1, 2, 5, 2)$
 \item $G_7(2, 3, 1, 3, 2, 1, 4, 3)$
 \item $G_7(2, 2, 2, 3, 2, 1, 4, 3)$
 \item $G_7(1, 3, 2, 3, 2, 1, 4, 3)$
 \item $G_7(2, 1, 3, 3, 2, 1, 4, 3)$
 \item $G_7(1, 2, 3, 3, 2, 1, 4, 3)$
 \item $G_7(1, 1, 4, 3, 2, 1, 4, 3)$
 \item $G_7(2, 2, 2, 3, 1, 2, 4, 3)$
 \item $G_7(2, 1, 3, 3, 1, 2, 4, 3)$
 \item $G_7(2, 2, 2, 2, 2, 2, 4, 3)$
 \item $G_7(1, 2, 3, 2, 2, 2, 4, 3)$
 \item $G_7(1, 3, 2, 3, 1, 1, 5, 3)$
 \item $G_7(1, 2, 3, 3, 1, 1, 5, 3)$
 \item $G_7(1, 1, 4, 3, 1, 1, 5, 3)$
 \item $G_7(1, 3, 2, 2, 1, 2, 5, 3)$
 \item $G_7(1, 2, 3, 2, 1, 2, 5, 3)$
 \item $G_7(1, 1, 3, 3, 1, 2, 5, 3)$
 \item $G_7(2, 2, 2, 2, 2, 1, 4, 4)$
 \item $G_7(2, 1, 3, 2, 2, 1, 4, 4)$
 \item $G_7(1, 1, 4, 2, 2, 1, 4, 4)$
 \item $G_7(2, 2, 2, 2, 1, 2, 4, 4)$
 \item $G_7(2, 1, 3, 2, 1, 2, 4, 4)$
 \item $G_7(1, 1, 3, 2, 2, 2, 4, 4)$
 \item $G_7(1, 2, 3, 2, 1, 1, 5, 4)$
 \item $G_7(1, 1, 4, 2, 1, 1, 5, 4)$
 \item $G_7(1, 2, 3, 1, 1, 2, 5, 4)$
 \item $G_7(1, 2, 2, 2, 1, 2, 5, 4)$
 \item $G_7(1, 1, 3, 2, 1, 2, 5, 4)$
 \item $G_7(1, 1, 4, 1, 1, 1, 5, 5)$
 \item $G_7(1, 1, 3, 1, 1, 2, 5, 5)$
\end{enumerate}
\end{multicols}
\normalsize

\section{Conclusion}\label{sec:conclusion}
We provided an alternate, stronger proof that there are only finitely many $k$-vertex-critical $(P_5,\text{ gem})$-free graphs for all $k$. We used our proof to completely classify all such graphs for $k\le 7$, extending the lists given in~\cite{CaiGoedgebeurHuang2021} for $k\le 5$. While our results give the most thorough structural description of vertex-critical $(P_5,\text{ gem}$)-free graphs, there are still a few open problems.

\begin{problem}\label{prob:8andup}
Classify all $k$-vertex-critical $(P_5,\text{ gem})$-free graphs for fixed $k\ge 8$.
\end{problem}

To our knowledge, complete lists of $k$-vertex-critical $(H_1,H_2)$-free graphs for $k\ge 8$ are only known for (gem, co-gem)-free graphs, where they are given for all $k\le 16$~\cite{AbuadasCameronHoangSawada2022}. Therefore, it is of interest to generate these complete lists for other families of graphs. Our methods from Section~\ref{sec:6and7} work in theory to find these classifications, however, limitations on computational resources prevent us from generating these lists for now. One approach to reduce the search space of our methods in Section~\ref{sec:6and7} is to show that $G_i$ has no $k$-vertex-critical clique expansions for some $1\le i\le 10$. In fact, for $k\le 7$, there are no clique expansions of any $G_i$ other than for $i=1$ or $i=7$ and we proved in Lemma~\ref{lem:nocritinG2G3} that there are none for $i=2$ and $i=3$ for any $k$. This leads to the following open problem:

\begin{problem}\label{prob:cliqueexpsofotherGis}
Determine if there are any $k$-vertex-critical clique expansions of $G_i$ for $i\neq 1$ and $i\neq 7$.
\end{problem}

If there are no such graphs for Open Problem~\ref{prob:cliqueexpsofotherGis}, then the $G_1$ clique expansions have already been classified from the recent work on (gem, co-gem)-free graphs~\cite{AbuadasCameronHoangSawada2022}, so to classify all $k$-vertex-critical $(P_5,\text{ gem})$-free graphs would require only searching all clique expansions of $G_7$. 

\section*{Acknowledgements}
This work was supported by the Canadian Tri-Council Research Support Fund. Both authors were each supported by individual NSERC Discovery Grants.

\bibliographystyle{abbrv}
\bibliography{refs}

\end{document}